\documentclass[a4paper]{amsart}
\usepackage{amsfonts,amsmath,amssymb,amsthm,cmtiup} 
\usepackage{mathrsfs}
\usepackage{cite}

\newtheorem{theorem}{Theorem}
\newtheorem*{theorem*}{Theorem}
\newtheorem{lemma}{Lemma} 
\newtheorem*{lemma*}{Lemma} 
 
\newtheorem*{corollary*}{Corollary}

\theoremstyle{definition} 
\newtheorem{remark}{Remark}
\newtheorem*{remark*}{Remark}

\newtheorem*{definition*}{Definition}

\newtheorem*{example*}{Example}

\newtheorem*{problem*}{Problem}

\newcommand{\ind}{\operatorname{ind}}

\def\cl#1{\skew3\overline{#1}}
\def\clx#1#2{\skew3\overline{#1}^{\raise.5pt\hbox{\scriptsize$\,#2$\!}}}

\def\clrx#1#2{\skew3\overline{#1\hspace{-0.1em}}^{\raise.5pt\hbox{\,\scriptsize$#2$}}}

\def\clax#1#2#3{\skew6\overline{#1_{\rlap{$\scriptstyle 
#2$}}}^{\raise.5pt\hbox{\hspace{0.2667em}\scriptsize$#3$}}\!
}
\def\clarx#1#2#3{\skew7\overline{\,#1_{\rlap{$\scriptstyle 
#2$}}}^{\raise.5pt\hbox{\hspace{0.2667em}\scriptsize$#3$}}\!
}
\def\Cl#1^#2{\cl{#1\kern-.09em}\kern.09em\vphantom{#1}^{#2}}

\newcommand\bR{\mathbb R}
\newcommand\bN{\mathbb N}

\newcommand\bZ{\mathbb Z}

\let\le\leqslant
\let\ge\geqslant

\begin{document}

\title{No Subgroup Theorem\\ for the Covering Dimension\\ of Topological Groups} 
\author{Ol'ga Sipacheva}

\begin{abstract}
A strongly zero-dimensional topological group containing a closed subgroup of positive 
covering dimension is constructed. 
\end{abstract}



\address{Department of General Topology and Geometry, Faculty of Mechanics and  Mathematics, 
M.~V.~Lomonosov Moscow State University, Leninskie Gory 1, Moscow, 199991 Russia}

\email{o-sipa@yandex.ru, osipa@gmail.com}

\maketitle

In \cite[Problem~2.1]{88} Shakhmatov asked whether the inequality\footnote{Recall that $\dim X$ ($\dim_0 X$) is 
the least integer $n\ge -1$ such that any finite open (cozero) cover of $X$ has a finite open (cozero) refinement 
of order $n$, provided that such an $n$ exists (see~\cite{Ch}). A space $X$ with $\dim_0 X=0$ (with $\ind X=0$) 
is said to be \emph{strongly zero-dimensional} (\emph{zero-dimensional}).} $\dim_0 H\le \dim_0 G$ holds for an arbitrary 
subgroup $H$ of an arbitrary topological group $G$. He proved that the answer is positive if $G$ is a locally 
pseudocompact or a Lindel\"of $\Sigma$ group~\cite{88, Sh}. It is also known that if $H$ is $\bR$-factorizable, 
then $\dim_0 H\le \dim_0 G$~\cite[Theorem~2.7]{Tk}. In particular, if $H$ is Lindel\"of, then $\dim H=\dim_0 H\le 
\dim_0 G$. In this paper we construct an $\bR$-factorizable group $G$ with $\dim_0 G=0$ which contains 
a closed subgroup $H$ of positive covering dimensions $\dim_0$ and $\dim$. Moreover, $H$ is the product of 
two Lindel\"of (and hence $\bR$-factorizable) strongly zero-dimensional groups, one of which is cosmic. 
Thus, as a byproduct, we obtain a negative solution of Tkachenko's Problem~4.1 in \cite{Tk}, as well as of  
Problems~8.2.2, 8.5.4, and~8.5.6 in~\cite{AT}. 

Let $C\subset [0,1]$ be the usual Cantor set, and let $\tau$ denote its topology.  In \cite{preprint} (see, in 
particular, Remark~1), based on the construction of Przymusi\'nski in \cite{Prz1} and \cite{Prz2} (see also 
\cite{Ch}), we defined a new topology $\tau_1\supset \tau$ on $C$ and a set $S_2\subset C$ such that the spaces 
$C_1=(C,\tau_1)$ and $C_2=(S_2, \tau|_{S_2})$ and their free Abelian topological groups $A(C_1)$ and $A(C_2)$  
have the following properties: 
\begin{enumerate} 
\item 
$C_1^n$ is Lindel\"of for any $n\in \bN$, and hence so is $A(C_1)$ (see \cite[Corollary~7.1.18]{AT}); 
\item 
$C_2$ is cosmic (that is, has a countable network), and hence so is $A(C_2)$ (see \cite[Corollary~7.1.17]{AT});  
\item 
$\dim_0 C_i=\dim_0 A(C_i)=0$ for $i=1,2$; 
\item 
$\dim_0 (A(C_1)\times A(C_2))>0$. 
\end{enumerate}

Recall that, given any continuous pseudometric $p$ on 
a Tychonoff space $X$, its \emph{Graev extension} to the free Abelian topological 
group $A(X)$ is defined. This is an invariant\footnote{A pseudometric $d$ on a group $G$ is \emph{invariant} if 
$d(x, y)=d(g\cdot x, g\cdot y)=d(x\cdot g, y\cdot g)$ for any $x,y,g\in G$. If $d$ is invariant, then we also 
have $d(x^{-1},y^{-1})=d(x,y)$ for any $x,y\in G$.} continuous pseudometric $\widehat p$ on $A(X)$ which, in 
particular, has the following properties \cite{Graev}, \cite{Marxen} (see also \cite[pp.~425--434]{AT}): 

\begin{enumerate} 
\item 
$\widehat p(x,y)=\min\{1, p(x,y)\}$ for any $x,y\in X$; 
\item 
for any $g=k_1x_1 + \cdots +k_nx_n \in A(X)$, where $k_i\in 
\bZ$ and $x_i\in X$ for $i\le n$, $\widehat p(g,\bf 0)$ is a sum of numbers in the set 
$\{1\}\cup \{p(x_i,x_j): i,j\le n\}$ (here and in what follows, $\bf 0$ denotes the zero element of $A(C_i)$);
\item
if $p$ is a metric, then so is $\widehat p$;
\item 
if $p_1$ and $p_2$ are continuous pseudometrics on $X$ and $p_2$ majorizes $p_1$, i.e., $p_1(x,y)\le 
p_2(x,y)$ for all $x, y\in X$, then $\widehat p_2$ majorizes $\widehat p_1$; 
\item
all open $\widehat p$-balls, where $p$ ranges over all continuous pseudometrics on $X$ bounded by 1, form a base 
for the topology of $A(X)$. 
\end{enumerate}

Fix an $i\in \{1,2\}$ and let $\mathscr P_i$ denote the set of all continuous pseudometrics bounded by 1 on 
$C_i$. 

\begin{lemma}
\label{lemma1}
Any continuous pseudometric $p$ on $C_i$ bounded by $1$ is majorized by a continuous ultrametric\footnote{An 
\emph{ultrametric} is a metric satisfying the \emph{strong triangle inequality} $d(x,z)\le \max\{d(x,y),d(y,z)\}$ 
for all $x$, $y$, and~$z$.} $d_p$ on $C_i$ which takes values in $\{0\}\cup \{1/{2^n}: n\in\bN\cup\{0\}\}$. 
\end{lemma}

\begin{proof}
We remember that the topology of $C_i$ is finer than that induced from $\bR$. Thus, setting $\bar 
p(x,y)=\max\{p(x,y), |x-y|\}$ for all $x, y\in C_i$, we obtain a continuous metric $\bar p$ on $C_i$ majorizing 
$p$. Below we recursively construct disjoint open covers $\gamma_n$, $n\in \omega$, of $C_i$ and use them to 
define a continuous ultrametric majorizing $\bar p$. 

We set $\gamma_0= C_i$. 

Suppose that $k\in \bN$ and disjoint open covers $\gamma_0,\gamma_1, \dots, \gamma_{k-1}$ are already defined so 
that each of them refines the preceding ones. Since $C_i$ is paracompact, there 
exists a locally finite open cover $\gamma$ of $C_i$ which refines both $\gamma_{k-1}$ and the cover of $C_i$ by 
open $\bar p$-balls of radius $\frac{1}{2^{k+1}}$. According to Lemma~17.16 of \cite{Ch}, $\gamma$ has a disjoint 
open refinement; we denote it by~$\gamma_k$. 

We have obtained a sequence $(\gamma_n)_{n\in \bN}$ of disjoint open covers of $C_i$ such that, for each $n\ge 
1$,  every element of $\gamma_n$ is contained in an open $\bar p$-ball of radius $\frac 1{2^{n+1}}$ and 
$\gamma_{n+1}$ is a refinement of $\gamma_n$. The function $d_p\colon C_i\times C_i\to \bR$ defined by 
$$ 
d_p(x,y)=\inf\Bigl\{\frac1{2^n}: \exists U\in \gamma_{n}\text{ such that }x,y\in U\Bigr\} 
$$
is a continuous ultrametric on $C_i$, and it majorizes~$\bar p$. 
\end{proof}

\begin{lemma}
\label{lemma2}
The group $A(C_i)$ embeds in a product of zero-dimensional second countable groups as a closed subgroup. 
\end{lemma}

\begin{proof}
Since $d_p$ is a metric for any continuous pseudometric $p$ on $C_i$, it follows that all Graev 
extensions $\widehat d_p$ are invariant metrics. Hence each of them determines a Hausdorff group topology 
$\tau_p$ on $A(C_i)$, and the topology of the free Abelian topological group $A(C_i)$ is the supremum of these 
topologies (see the property (5) of Graev extensions). Therefore, the identity isomorphisms $i_p\colon A(C_i)\to (A(C_i),\tau_p)$, $p\in \mathscr P_i$, 
separate points from closed sets, so that their diagonal 
$\mathop{\large\Delta}\limits_{p\in \mathscr P_i}i_p$ is a homeomorphism between $A(C_i)$ 
and the diagonal of the product $\prod\limits_{p\in \mathscr P_i}(A(C_i),\tau_p)$ of metrizable 
topological groups. Clearly, this homeomorphism is a topological isomorphism. Thus, $A(C_i)$ is topologically 
isomorphic to a closed subgroup of $\prod\limits_{p\in \mathscr P_i}(A(C_i),\tau_p)$. 

In view of the property (2) of 
Graev extensions, each metric $\widehat d_p$ takes only rational values. Since the $\widehat d_p$-balls of 
irrational radii form a base of $\tau_p$ and are clopen, it follows that the metrizable space $(A(C_i), \tau_p)$ 
is zero-dimensional for each $p\in \mathscr P_i$. Moreover, it is Lindel\"of, being a continuous image of the 
Lindel\"of space $A(C_i)$; therefore, it is second countable. 
\end{proof}

\begin{theorem}
There exists a strongly zero-dimensional Abelian topological group $G$ containing a closed subgroup $H$ 
with $\dim_0(H)>0$. 
\end{theorem}

\begin{proof}
It suffices to set $G=\prod\limits_{p\in \mathscr P_1}(A(C_1),\tau_p)\times \prod\limits_{p\in \mathscr 
P_2}(A(C_2),\tau_p)$ and $H=A(C_1)\times A(C_2)$ and recall that any product of zero-dimensional second countable 
spaces is strongly zero-dimensional~\cite[Theorem~3]{Morita}. 
\end{proof}

\begin{remark}
It is easy to construct a similar example for the \v Cech--Lebesgue covering dimension 
$\dim$ (but the subgroup $H$ cannot be made closed in this case, because, for the dimension $\dim$, the closed 
subset theorem holds \cite[Proposition~2.11]{Ch}). Indeed, consider the Sorgenfrey plane $S\times S$. It is 
zero-dimensional and hence embeds in the Cantor cube $K=\{0,1\}^{2^\omega}$ \cite[Theorem~6.2.16]{Eng}. Since 
$K$ is strongly zero-dimensional, it follows that the free Abelian topological group $A(K)$ is zero-dimensional 
\cite{Tk3}, and since $A(K)$ is Lindel\"of \cite[Corollary~7.1.18]{AT}, it follows that $\dim A(K)=0$. Let $H$ 
denote the subgroup of $A(K)$ generated by $S\times S$. Clearly, $H\cap K=S\times S$, and hence $S\times S$ is 
closed in $H$. We have $\dim (S\times S)>0$, because all spaces with $\dim$ zero are normal, while $S\times S$ is 
not. Therefore, $\dim H>0$. 
\end{remark}

We conclude with an observation concerning $\bR$-factorizable groups. 
A topological group G is said to be \emph{$\bR$-factorizable} if, for every continuous function $f\colon G\to 
\bR$, there exists a continuous epimorphism $h\colon G \to H$ onto a second countable topological group $H$ and a 
continuous function $g\colon H\to \bR$ such that $f = g \circ h$. It is known that any Lindel\"of group is 
$\bR$-factorizable \cite[Assertion~1.1]{Tk4} and that $\dim_0 H\le \dim_0 G$ for any topological group $G$ and 
any $\bR$-factorizable subgroup $H$ of $G$. These two facts imply that both groups $A(C_1)$ and $A(C_2)$ are 
$\bR$-factorizable, while their product is not. Thus, the following theorem holds.

\begin{theorem}
There exist Abelian topological groups $G$ and $H$ with the following properties:
\begin{enumerate}
\item
$G^n$ is Lindel\"of for each $n\in \bN$ and $H$ is cosmic (and hence both $G$ and $H$ are $\bR$-factorizable); 
\item 
both $G$ and $H$ are Raikov complete (and hence so is $G\times H$);
\item
the product $G\times H$ is not $\bR$-factorizable.
\end{enumerate}
\end{theorem}

\begin{proof} 
We set $G=A(C_1)$ and $H=A(C_2)$. 
As shown above, these groups $G$ and $H$ satisfy conditions (1) and (3), 
and (2) follows from the fact that any paracompact space is Dieudonn\'e 
complete (see \cite[8.5.13]{Eng}) and the Tkachenko--Uspenskij theorem that the free Abelian topological group of 
a Dieudonn\'e complete space is Raikov complete \cite{Tk2}, \cite{Usp}. 
\end{proof}

The author is very grateful to Evgenii Reznichenko for fruitful discussions.

\end{document}